\documentclass [12pt,oneside]{amsart}
\usepackage{amsmath}

\usepackage{multicol}
\usepackage {amsfonts}
\usepackage {amsmath}
\usepackage {amsthm}
\usepackage {amssymb}
\usepackage {framed}
\usepackage {amsxtra}
\usepackage {enumerate}
\usepackage {graphicx}
\usepackage{color}
\usepackage{graphicx}
\usepackage{graphicx}
\usepackage{wrapfig}
\usepackage{lscape}
\usepackage{rotating}
\usepackage{epstopdf}
\usepackage{pdflscape}
\usepackage{setspace}
\usepackage[left=0.86in, right=.8 in, top=0.7 in, footskip=0.5 in, bottom=0.8 in]{geometry}
\usepackage{adjustbox}
\usepackage{caption}
\makeatletter

\theoremstyle{definition}
\newtheorem{df}{Definition} [section]

\theoremstyle{plain}
\newtheorem{thm}[df]{Theorem}

\newtheorem{lemma}[df]{Lemma}

\newtheorem{conj}[df]{Conjecture}

\newtheorem{remark}[df]{Remark}

\title{On second order linear sequences of composite numbers}

\author{Dan Ismailescu$^2$}
\thanks{$^1$Mathematics Department, Hofstra University, Hempstead, NY 11549, \texttt{dan.p.ismailescu@hofstra.edu}}

\author{Adrienne Ko$^1$}
\thanks{$^2$ Fieldston School, Bronx, NY 10471, \texttt{19ako@ecfs.org}}
\author{Celine Lee $^3$}
\thanks{$^3$Chinese International School, Hong Kong SAR, China, \texttt{celinl@student.cis.edu.hk}}

\author{Jae Yong Park$^4$}
\thanks{$^4$The Lawrenceville School, Lawrenceville, NJ 08648, \texttt{jpark19@lawrenceville.org}}

\begin{document}


\begin{abstract}
In this paper we present a new proof of the following 2010 result of Dubickas, Novikas, and \u{S}iurys:

{\bf Theorem.}
{\emph {Let $(a,b)\in \mathbb{Z}^2$ and let $(x_n)_{n\ge 0}$ be the sequence defined by some initial values $x_0$ and $x_1$ and the second order linear recurrence
\begin{equation*}
x_{n+1}=ax_n+bx_{n-1}
\end{equation*}
for $n\ge 1$. Suppose that $b\neq 0$ and $(a,b)\neq (2,-1), (-2, -1)$. Then there exist two relatively prime positive integers $x_0$, $x_1$ such that $|x_n|$ is a composite integer
for all $n\in \mathbb{N}$.}}

The above theorem extends a result of Graham who solved the problem when $(a,b)=(1,1)$.




\end{abstract}

\maketitle
\pagenumbering{arabic}

\section{\bf{Introduction}}

In this paper we present a new proof of the following result of Dubickas, Novikas, and \u{S}iurys:
\begin{thm}\label{mainthm}\cite{dubickas}
Let $(a,b)\in \mathbb{Z}^2$ and let $(x_n)_{n\ge 0}$ be the sequence defined by some initial values $x_0$ and $x_1$ and the second order linear recurrence
\begin{equation}\label{recurrence}
x_{n+1}=ax_n+bx_{n-1}
\end{equation}
for $n\ge 1$. Suppose that $b\neq 0$ and $(a,b)\neq (2,-1), (-2, -1)$. Then there exist two relatively prime positive integers $x_0$, $x_1$ such that $|x_n|$ is a composite integer
for all $n\in \mathbb{N}$.
\end{thm}

Throughout the paper we will use the following convention:  a nonnegative integer $n$ is said to be \emph{composite} if $n\ne 0, 1$ and $n$ is not a prime number.

Graham \cite{graham} considered the problem above in the particular case $(a,b)=(1,1)$. He found two relatively prime positive integers $x_0$ and $x_1$ such that the sequence
$x_{n+1}=x_n+x_{n-1}$, consists of composite numbers only. Graham's starting pair is
\begin{equation*}
(x_0,x_1)=(331635635998274737472200656430763, \, 1510028911088401971189590305498785).
\end{equation*}
Graham's technique was successively refined by Knuth \cite{knuth}, Wilf \cite{wilf} , and Nicol \cite{nicol} who found smaller pairs $(x_0,x_1)$. The current record is due to Vsemirnov \cite{vsemirnov}
\begin{equation}\label{vsemirnov}
(x_0, x_1)= (106276436867, 35256392432).
\end{equation}
The above results are based on the fact that the Fibonacci sequence is a \emph{divisibility sequence}, that is, $F_n|F_m$ whenever $n|m$, and on finding a finite
\emph{covering system of congruences} $r_i \pmod {m_i}, 1\le i\le t$, such that there exist distinct primes $p_1, p_2,\ldots p_t$ so that $p_i|F_{m_i}$ for all $i=1, 2, \ldots, t$.

As Dubickas et al. we use these techniques to prove Theorem \ref{mainthm} when $|b|=1$ although our covering systems are different. For more details we refer the reader to Section  5.

We summarize our plan for proving Theorem \ref{mainthm} as follows.

In Section 2 we prove three easy lemmata which will be useful later on. We believe that Lemma \ref{lemma1} is of independent interest.

In Section 3 we study two simple cases: $(i)\,\, a=0$ and $(ii)\,\, a^2+4b=0$. In this section we also show that the condition
$(a,b)\ne (\pm 2, -1)$ is necessary.

Section 4 deals with the case $|b|\ge 2$. Following \cite{dubickas}, we will choose $x_1\equiv 0 \pmod{|b|}$, since then \eqref{recurrence} implies that $x_n\equiv 0 \pmod{|b|}$ for all $n\ge 2$. The main difficulty is to show that $x_0$ and $x_1$ can be chosen such that $x_n\neq -b, 0, b$ for every $n\ge 0$.

In this case, Dubickas et al. present a mainly existential proof which relies on a series of six lemmata. In contrast, our proof is constructive as we provide explicit expressions for $x_0$ and $x_1$ as simple polynomials of $a$ and $b$.

In Section 5 we consider the case $|b|=1$. Dubickas, Novikas and \u{S}iurys prove that except for finitely many values of $|a|$ one can take $x_0$ and $x_1$ so that each $x_n$ is divisible by one of five distinct appropriately chosen prime numbers. We prove that, with the exception of finitely many values of $|a|$,  four primes suffice.

The proof in the case $|b|=1$ relies on the divisibility properties of the \emph{Lucas sequence of the first kind} $(u_n)_{n\ge0}$ defined as
\begin{equation}\label{lucas}
u_0=0, u_1=1 \quad \text{and} \quad u_{n+1}=au_n+bu_{n-1}, \,\,\text{for} \,\, n\ge 1.
\end{equation}
Finally, in Section 6 we prove that if $|a|\ge 3$ and $b=-1$ then $u_n$ is composite for all $n\ge 3$. The interesting fact is that in this case it appears that there is no finite set of prime numbers $p_1, p_2,\ldots p_t$ such that each $u_n$ is divisible by some $p_i$, $i=1, 2, \ldots t$.

\section{\bf {Three useful lemmata}}

\begin{lemma}\label{lemma1}
Consider the sequence $(x_n)_{n\ge 0}$ given by \eqref{recurrence}. Then,
\begin{equation}\label{eq1}
x_{n+1}^2-ax_nx_{n+1}-bx_n^2=(-b)^n(x_1^2-ax_0x_1-bx_0^2)
\end{equation}
\end{lemma}

\begin{proof}
	\begin{gather*}
		\begin{bmatrix}
		x_{n+2} & x_{n+1}\\
		x_{n+1} & x_{n}
		\end{bmatrix}
		=
		\begin{bmatrix}
		ax_{n+1}+bx_n & ax_n+bx_{n-1}\\
		x_{n+1} & x_{n}
		\end{bmatrix}
		=
		\begin{bmatrix}
		a & b\\
		1 & 0
		\end{bmatrix}
		\cdot
		\begin{bmatrix}
		x_{n+1} & x_{n}\\
		x_{n} & x_{n-1}
		\end{bmatrix}
		=
\\
\smallskip
=
		\begin{bmatrix}
		a & b\\
		1 & 0
		\end{bmatrix}
		\cdot
		\begin{bmatrix}
		a & b\\
		1 & 0
		\end{bmatrix}
		\cdot
		\begin{bmatrix}
		x_{n} & x_{n-1}\\
		x_{n-1} & x_{n-2}
		\end{bmatrix}
		=
		\begin{bmatrix}
		a & b\\
		1 & 0
		\end{bmatrix}^2
		\cdot
		\begin{bmatrix}
		x_{n} & x_{n-1}\\
		x_{n-1} & x_{n-2}
		\end{bmatrix}
		= \cdots =
		\begin{bmatrix}
		a & b\\
		1 & 0
		\end{bmatrix}^n
		\cdot
		\begin{bmatrix}
		x_2 & x_1\\
		x_1 & x_0
		\end{bmatrix}
	\end{gather*}

\smallskip
	
Taking determinants on both sides we obtain	
	\begin{align*}
		&det\begin{bmatrix}
		x_{n+2} & x_{n+1}\\
		x_{n+1} & x_{n}
		\end{bmatrix}
		=
		det\begin{bmatrix}
		a & b\\
		1 & 0
		\end{bmatrix}^n
		\cdot
		det\begin{bmatrix}
		x_2 & x_1\\
		x_1 & x_0
		\end{bmatrix}
		\longrightarrow
		x_{n+2}\cdot x_n-x_{n+1}^2=(-b)^n\cdot (x_2x_0-x_1^2)
		\longrightarrow\\
		&\longrightarrow(ax_{n+1}+bx_n)x_n-x_{n+1}^2=(-b)^n\cdot ((ax_1+bx_0)x_0-x_1^2), \quad\text{which after expanding becomes}\\
		&x_{n+1}^2-ax_nx_{n+1}-bx_n^2=(-b)^n(x_1^2-ax_0x_1-bx_0^2), \quad \text{as claimed.}
	\end{align*}
	
\end{proof}

\begin{lemma}\label{lemma2}
Consider the sequence $(x_n)_{n\ge 0}$ given by \eqref{recurrence}. Suppose that $1\le |x_0|<|x_1|$ and $|a|>|b|\ge 1$. Then the sequence $(|x_n|)_{n\ge 0}$ is strictly increasing.
\end{lemma}
\begin{proof}
We use induction on $n$. By hypothesis, the statement is true for $n=0$.

Suppose that $|x_{n-1}|<|x_n|$ for some $n\ge 1$. We intend to prove that $|x_n|<|x_{n+1}|$. Indeed
\begin{align*}
|x_{n+1}|&=|ax_n+bx_{n-1}|\ge |ax_n|-|bx_{n-1}| = |a||x_n|-|b||x_{n-1}|>\\
&>|a||x_n|-|b||x_n|=(|a|-|b|)|x_n|,\,\, \text{by the induction hypothesis}
\end{align*}

and since $|a|-|b|\ge 1$ we finally have $|x_{n+1}|> |x_n|$, which completes the induction.
\end{proof}

\begin{lemma}\label{lemma3}
Let $n_1$, $n_2$ and $n_3$ be three positive integers such that no prime number $p$ divides all of them. Then, there exists an integer $k\ge 2$ such that $n_1$ and $n_2+k n_3$ are relatively prime.
\end{lemma}

\begin{proof}
Let $d:=\gcd(n_2,n_3)$. Note that $\gcd(n_1,d)=1$ otherwise $d$ divides $n_1$, $n_2$ and $n_3$.

By Dirichlet's theorem on arithmetic progressions there exists a $k$ such that $n_2/d+k n_3/d$ is a prime number greater than $n_1$.
Then $d\,(n_2/d+k\cdot n_3/d) = n_2+k\:n_3$ is both greater and relatively prime to $n_1$.
\end{proof}

\section{\bf Two simple special cases: $(i)\,\, a=0$ and  $(ii)\,\, a^2+4b=0$.}

\emph {Case (i).} Since $a=0$, it can be easily proved that $x_{2n}=b^nx_0$ and $x_{2n+1}=b^nx_1$. It suffices to take $x_0=4$ and $x_1=9$ to obtain that $x_n$ is composite for all $n\ge 0$.

\emph {Case (ii).} If $a^2+4b=0$, then $a$ must be even, $a=2c$ and therefore $b=-a^2/4=-c^2$. We divide the proof into two cases: $|b|\ge 2$ and $b=-1$.

If $|b|\ge 2$ we have $|c|\ge 2$. Let us now take $x_0=4c^2-1$ and $x_1=2c^3$. Then $x_0$ and $x_1$ are relative prime positive composite integers. One immediately obtains $x_2=c^2$ thus $x_2$ is also composite. Also, it can be easily proved that $x_n=c^n\left((n-1)-(2n-4)c^2\right)$ for all $n\ge 3$.

Note that for any $n\ge 3$ one cannot have $x_n = 0$, otherwise $4\le c^2=(n-1)/(2n-4)\le 1$, which is impossible.
It follows that $|x_n|$ is composite for all $n\ge 3$.

If $b=-1$ then $|a|=2$. If $a=2$ then a simple induction shows that $x_{n+1}=(n+1)x_1-nx_0=x_1+n(x_1-x_0)$ for all $n\ge 0$, that is, $(x_n)_{n\ge 0}$ is an arithmetic sequence whose first term and common difference are relatively prime. By Dirichlet's theorem on primes in arithmetic progressions, it follows that $|x_n|$ is a prime number for infinitely many values of $n$.
If $a=-2$ then one can show that $x_{n+1}=(-1)^n(x_1+n(x_0+x_1))$. In this case, $(x_n)_{n\ge 0}$ is the union of two arithmetic sequences both of which have the first term and the common difference relatively prime. Again, for any choice of $x_0$ and $x_1$ relatively prime, $|x_n|$ is a prime for infinitely many $n$. This proves the necessity of the condition $(a,b)\neq(\pm 2, -1)$.

\section{{\bf The case $|b|\ge 2$}}
Based on the results from the previous section, from now on one can assume that
\begin{equation}
|a|\ge 1 \quad \text {and}\quad a^2+4b\ne 0.
\end{equation}
We divide the proof into three separate cases:

{\bf Case I:} {$|a|>|b|$}, {\bf Case II:} {$1\le |a|\le |b|$ and $|b|$ is composite,  and {\bf Case III:} {$1\le |a|\le |b|$ and $|b|$ is a prime.

{\bf{Case I: $|a|>|b|$}}

To complete the proof of Theorem \ref{mainthm} in this case we take $x_0=b^4-1, x_1=b^4$. Clearly $x_0$ and $x_1$ are both positive, relatively prime and composite. Moreover, $x_0<x_1$. Then $x_n\equiv 0\pmod{b}$ for all $n\ge1$ and by using Lemma \ref{lemma2} it follows that $(|x_n|)_{n\ge 0}$ is strictly increasing. Hence, we have that $|x_n|\ge x_0>|b|$ for all $n\ge 0$ and therefore, $|x_n|$ is composite for all $n\ge 0 $.

\smallskip

{\bf{Case II: $1\le |a|\le |b|$ and $|b|$ is composite}}

Choose $x_0=4b^4-1$, $x_1=2b^2$.
Clearly, $\gcd(x_0,x_1)=1$ and $x_0, x_1$ are positive composite integers. It also follows that $x_n\equiv 0\pmod{b}$ for all $n\ge1$ and since $|b|$ is composite it follows that $|x_n|$ is composite, unless $x_n=0$ for some $n$. We will prove that this cannot happen.

For the sake of contradiction, suppose that $x_{n+1}=0$ for some $n\ge 2$. Then, by Lemma \ref{lemma1}, we have $x_n^2=(-b)^{n-1}(x_1^2-ax_0x_1-bx_0^2)$.

Since $x_0=4b^4-1$ and $x_1=2b^2$ we have $x_1^2-ax_0x_1-bx_0^2=(-b)(16b^8+8ab^5-8b^4-4b^3-2ab+1)$, \text{which implies that}
\begin{equation}\label{xnsquared}
x_n^2=(-b)^n(16b^8+8ab^5-8b^4-4b^3-2ab+1).
\end{equation}
If $n$ is even, this implies that $16b^8+8ab^5-8b^4-4b^3-2ab+1$ is a perfect square. If $n$ is odd, this implies that $(-b)(16b^8+8ab^5-8b^4-4b^3-2ab+1)$ is a perfect square. We will prove that none of these are possible if $a^2+4b\ne0$.
	
Note first that since $|a|\le |b|$ we have
\begin{equation}\label{cII}
(4b^4+ab-2)^2<16b^8+8ab^5-8b^4-4b^3-2ab+1<(4b^4+ab)^2
\end{equation}
Indeed, these inequalities are equivalent to the following ones
\begin{equation}\label{cIIa}
8b^4-2-(4b^3+(ab-1)^2)>0, \quad \text{and}\quad 8b^4+4b^3-2+(ab+1)^2>0.
\end{equation}
 We have
\begin{equation}
|a|\le |b| \rightarrow |ab|\le b^2 \rightarrow |ab - 1|\le |ab|+1\le b^2+1 \rightarrow (ab- 1)^2\le b^4+2b^2+1.
\end{equation}
From this we obtain that
\begin{equation}
|4b^3+(ab-1)^2|\le 4|b|^3 + (ab-1)^2 \le b^4+4|b|^3+2b^2+1.
\end{equation}
To prove the first inequality in \eqref{cIIa} note that
\begin{align*}
&8b^4-2-(4b^3+(ab-1)^2)\ge 8b^4-2-|4b^3+(ab-1)^2|\ge\\
&\ge 8b^4-2-(b^4+4|b|^3+2b^2+1)= 7b^4-4|b|^3-2b^2-3=\\
&=7|b|^4-4|b|^3-2|b|^2-3=4|b|^3(|b|-1)+2|b|^2(|b|^2-1)+(|b|^4-3)\ge\\
&\ge 4|b|^3+2|b|^2+|b|^4-3\ge 4\cdot 2^3+2\cdot 2^2+2^4-3>0.
\end{align*}
The second inequality in \eqref{cIIa} is much easier to prove:
\begin{equation*}
8b^4+4b^3-2+(ab+1)^2\ge 8b^4+4b^3-2\ge 8|b|^4-4|b|^3-2\ge 8\cdot 2^4-4\cdot 2^3-2>0.
\end{equation*}
Hence, the inequalities \eqref{cII} hold true.
	
If the middle term in \eqref{cII} were to be a perfect square, the only option remaining is
\begin{equation*}
16b^8+8ab^5-8b^4-4b^3-2ab+1=(4b^4+ab-1)^2,\,\,\text{which implies}\,\, b^2(a^2+4b)=0.
\end{equation*}
However, this is impossible if $a^2+4b\ne 0$. Hence, if one assumes that  $|a|\le |b|$ and $a^2+4b\ne 0$ then $16b^8+8ab^5-8b^4-4b^3-2ab+1$ cannot be a perfect square.
	
Suppose next that $(-b)(16b^8+8ab^5-8b^4-4b^3-2ab+1)$ is a perfect square. Since the two factors are mutually prime, it follows that both have to be perfect squares. In particular, the second one has to be a perfect square, and we have already proved that it cannot happen. It follows that the sequence $(x_n)$ does not contain any terms equal to $0$. This completes the proof of Theorem \ref{mainthm} in case II.

\bigskip

{\bf{Case III: $1\le |a|\le |b|$ and $|b|$ is a prime}}

We divide the proof of this case into three subcases: $|a|=1$, $|a|=|b|$ and $2\le |a|<|b|$.

{\bf{Subcase IIIa:  $|a|=1$, $|b|$ is a prime}}

If $a=1$ and $b>0$ then take $x_0=(2b^2-1)^2, x_1=b(b^2-1)$. Clearly, $x_0$ and $x_1$ are positive, relatively prime, composite integers. It follows immediately that $x_2=b^3(4b^2-3)$ is also composite. Finally, an easy induction shows that $x_n\equiv 0\pmod{b^2}$ and that $\frac{x_n}{b^2} \equiv -1\pmod{b}$ for all $n\ge 3$. Hence, $|x_n|$ is necessarily composite for all $n\ge 3$.

The other situations can be dealt with similarly.

If $a=-1$ and $b<0$ take $x_0=(2b^2-1)^2, x_1=-b(b^2-1)$. Then $x_2=b^3(4b^2-3)$ and for $n\ge 3$ one can show that $x_n\equiv 0\pmod{b^2}$ and \(\frac{x_n}{b^2}\)$\equiv (-1)^{n+1}\pmod{b}$. Again, all $|x_n|$ are composite.

If $a=1$ and $b<0$ we take $x_0=(2b^2-1)^2, x_1=-b(b^2+1)$. Then $x_2=b^3(4b^2-5)$ and for all $n\ge 3$ one can prove that $x_n\equiv 0\pmod{b^2}$ and \(\frac{x_n}{b^2}\)$\equiv -1\pmod{b}$. Hence, all $|x_n|$ are composite.

Finally, if $a=-1$ and $b>0$ select $x_0=(2b^2-1)^2, x_1=b(b^2+1)$. Then $x_2=b^3(4b^2-5)$ and for all $n\ge 3$ we have that $x_n\equiv 0\pmod{b^2}$ and \(\frac{x_n}{b^2}\)$\equiv (-1)^{n+1}\pmod{b}$. All $|x_n|$ are composite.

This completes the proof of Theorem in subcase IIIa.

\smallskip

{\bf{Subcase IIIb:  $|a|=|b|$, $|b|$ is a prime}}

If $a=b$, take $x_0=4b^4-1, x_1=2b^2$. Then $x_2=b(4b^4+2b^2-1)$ is composite and $x_n\equiv 0\pmod{b^2}$ for all $n\ge3$. It remains to show that $x_n\ne0$ for all $n$.

As in Case II, $x_{n+1}=0$ implies that either $16b^8+8b^6-8b^4-4b^3-2b^2+1$ or $(-b)(16b^8+8b^6-8b^4-4b^3-2b^2+1)$ is a perfect square. One can use the same technique as in Case II to prove that this cannot happen if $a^2+4b=b^2+4b\ne0$.

If $a=-b$, take $x_0=4b^4-1, x_1=2b^2$ (same choice). Then $x_2=b(4b^4-2b^2-1)$ is composite and $x_n\equiv 0\pmod{b^2}$ for all $n\ge3$. It remains to show that $x_n\ne0$ for all $n$.

As in Case II, $x_{n+1}=0$ implies that either $16b^8-8b^6-8b^4-4b^3+2b^2+1$ or $(-b)(16b^8-8b^6-8b^4-4b^3+2b^2+1)$ is a perfect square. The same approach as in Case II shows that this is impossible if $a^2+4b=b^2+4b\ne0$.

\smallskip

{\bf{Subcase IIIc:  $2\le|a|<|b|$, $|b|$ is a prime}}

It follows that $\gcd(a,b)=1$.

If $a>0, b>0$, take $x_0=a^3, x_1=b(b^2-a^2)$. Then $x_2=ab^3$. For $n\ge3$, $x_n\equiv 0\pmod{b^2}$ and \(\frac{x_n}{b^2}\)$\equiv -a^{n-1}\pmod{b}$. Since $\gcd(a,b)=1$, $x_n\ne0$.

If $a<0, b<0$, take $x_0=-a^3, x_1=-b(b^2-a^2)$. Then $x_2=-ab^3$. For $n\ge3$, $x_n\equiv 0\pmod{b^2}$ and \(\frac{x_n}{b^2}\)$\equiv a^{n-1}\pmod{b}$. Hence, $x_n\ne0$ for all $n$.

If $a<0, b>0$, take $x_0=-a^3, x_1=b(b^2+a^2)$. Then $x_2=ab^3$ and for $n\ge3$, $x_n\equiv 0\pmod{b^2}$ and \(\frac{x_n}{b^2}\)$\equiv a^{n-1}\pmod{b}$. So, $x_n\ne0$ for all $n$.

If $a>0, b<0$, take $x_0=a^3, x_1=-b(b^2+a^2)$. Then $x_2=-ab^3$. For $n\ge3$, $x_n\equiv 0\pmod{b^2}$ and \(\frac{x_n}{b^2}\)$\equiv -a^{n-1}\pmod{b}$. Since $\gcd(a,b)=1$, it follows that $x_n\ne0$ for all $n$.

Hence, in each case one can choose $x_0$ and $x_1$ positive, composite and relatively prime so that $x_2$ is composite, and for all $n\ge 3$ we have that $x_n\equiv 0\pmod{b^2}$ and $x_n\neq 0$. It follows that $|x_n|$ is composite for all $n\ge 0$ and this completes the proof of Theorem \ref{mainthm} in the case $|b|\ge 2$.

\section{\bf {The case $|b|$=1}}

The main idea behind the proof of Theorem \ref{mainthm} in this particular case can be summarized as follows: we want to find a finite set of primes $p_1, p_2, \ldots, p_t$ such that
for every $n\ge 0$ we have that $|x_n|$ is divisible by at least one of these primes.

We start the analysis with the simple case when $|a|$ has at least two distinct prime factors.
\begin{lemma}
Let $|b|=1$ and suppose that $|a|$ has at least two distinct prime factors: $p_1$ and $p_2$ with $p_1<p_2$. Then, the sequence given by \eqref{recurrence} and the initial terms $x_0=p_1^2$, $x_1=p_2^2$ satisfies the conditions in Theorem \ref{mainthm}.
\end{lemma}
\begin{proof} Clearly, $x_0$ and $x_1$ are positive relatively prime composite integers. An easy induction shows that $x_n\equiv 0 \pmod{p_1}$ is $n$ is even and $x_n\equiv 0 \pmod{p_2}$ if $n$ is odd. Since $|x_0|<|x_1|$ and $1=|b|<|a|$ the hypotheses of Lemma \ref{lemma2} are satisfied hence $(|x_n|)_{n\ge 0}$ is strictly increasing. It follows that $x_n$ is composite for all $n\ge 0$ and therefore Theorem \ref{mainthm} is verified in this case.
\end{proof}
\begin{remark}
For the remainder of the paper we will assume that $|a|=p_1^s$ for some prime $p_1$ and some nonnegative integer $s\ge 0$. Note that this allows the possibility that $|a|=1$.
\end{remark}
 Next, we introduce a special sequence $(u_n)_{n\ge 0}$ given by
\begin{equation}\label{lucassequence}
u_0=0, u_1=1,\quad  u_{n+1}=au_n+bu_{n-1} \quad \text{for all} \,\,n\ge 1.
\end{equation}
 This sequence is called the \emph{Lucas sequence of the first kind}.

One can prove that $(u_n)_{n\ge 0}$ is a \emph{divisibility sequence}, that is, $u_m$ divides $u_n$ whenever $m$ divides $n$.
Indeed, suppose that $a^2+4b\neq 0$ which implies that the roots $\alpha$ and $\beta$ of the characteristic equation are distinct $\alpha\neq \beta$.
Assume that $n=mk$. Then
\begin{align*}
u_n=\frac{\alpha^n-\beta^n}{\alpha-\beta}= \frac{\alpha^{mk}-\beta^{mk}}{\alpha-\beta}&= \frac{\alpha^m-\beta^m}{\alpha-\beta}\cdot\left(\alpha^{m(k-1)}+\alpha^{m(k-2)}\beta^m+\ldots+\beta^{m(k-1)}\right)\\
&=u_m\left(\alpha^{m(k-1)}+\alpha^{m(k-2)}\beta^m+\ldots+\beta^{m(k-1)}\right).
\end{align*}
The second factor of the last term in the equality above is a symmetric function in $\alpha$ and $\beta$ and therefore it can be written as a polynomial function of $\alpha+\beta=a$ and $\alpha\beta=-b$. It follows that $u_n/u_m$ is an integer as claimed.

Of particular interest in the sequel are the values of $u_4$ and $u_6$.
\begin{equation}\label{u4u6}
u_4=a(a^2+2b)\quad {\text and} \quad u_6=a(a^2+b)(a^2+3b).
\end{equation}
The next two results give information regarding the prime factors of $u_4$ and $u_6$ when $|b|=1$.

\begin{lemma}\label{bminus1} Suppose that $b=-1$, $|a|\ge 4$ and  $|a|=p_1^s$ for some prime $p_1$ and some $s\ge 1$. Then $u_6=a(a^2-1)(a^2-3)$ has at least three additional distinct prime factors $p_2, p_3$ and $p_4$, all three different from $p_1$.
\end{lemma}
\begin{proof}
Note that $a^2-3$ is not divisible by 4 or 9. Since $a^2-3\ge 13$ it follows that $a^2-3$ has an odd prime factor $p_4$ different from 3, and since $\gcd(a, a^2-3)\in\{1,3\}$ we have that $p_4\neq p_1$.
Clearly, neither $p_1$ nor $p_4$ are factors of $a^2-1$. Note that $a^2-1$ has at least two distinct prime factors. Indeed, is one assumes the opposite then both $|a-1|$ and $|a+1|$ are powers of some prime. But this cannot happen since $|a|\ge 4$. It follows that $a^2-1$ has at least two distinct prime factors $p_2$ and $p_3$ both different from $p_1$ and $p_4$.
\end{proof}
To illustrate Lemma \ref{bminus1} let us take $(a,b)=(-9,-1)$. Then $a^2-3=78=2\cdot3\cdot13$ while $a^2-1=80=2^4\cdot 5$. Hence, one has $p_1=3$, $p_4=13$, $p_2=2$, and $p_3=5$.

\begin{lemma}\label{bplus1} Suppose that $b=1$, $|a|\ge 6$ and $|a|=p_1^s$ for some prime $p_1$ and some $s\ge 1$.

If $p_1\neq 3$ then $u_4=a(a^2+2)$ has at least two additional distinct prime factors $p_2$ and $p_3$ both different from $p_1$.

If $|a|=3^s$ for some $s\ge 2$ then $u_6=a(a^2+1)(a^2+3)$ has at least three additional distinct prime factors $p_2$, $p_3$ and $p_4$ all three different from $p_1$.
\end{lemma}
\begin{proof}
Suppose first that $p_1=2$, that is, $|a|=2^s$ for some $s\ge 3$. Note that $a^2+2\equiv 0 \pmod 3$, hence one can choose $p_2=3$. Note that $a^2=2(2^{2s-1}+1)$ must have a prime factor different from $2$ and $3$. Indeed, if one assumes the opposite, then $2^{2s-1}+1=3^t$ for some $t\ge 2$. However, under the assumption that $2s-1\ge 2$ and $t\ge 2$ the Catalan-Mih\u{a}ilescu theorem implies that the only solution of this equation is $s=2, t=2$. But this implies that $|a|=2^2=4$, contradiction.

Consider next the case $|a|=p_1^s$ with $p_1>3$. In particular, $a^2+2$ is odd and therefore $\gcd(a, a^2+2)=1$. Moreover, since $a^2+2 \equiv 0 \pmod 3$ one can safely take $p_2=3$.
We claim that $a^2+2$ has at least one other prime factor $p_3$ different from $p_1$ and $p_2$. Indeed, if one assumes otherwise then $a^2+2= p_1^{2s}+2=3^t$ for some $t\ge 2$.
However, it was shown by Ljunggren \cite{ljunggren} that the more general equation $x^2+2= y^n$, $n\ge 2$ has the unique solution $x=5, y=3, n=3$. This would give $|a|=5$ which we excluded from our analysis.

Finally, suppose that $p_1=3$ and therefore $|a|=3^s$ for some $s\ge 2$. Then both $a^2+1$ and $a^2+3$ are even so one can take $p_2=2$. Note that $a^2+1=9^s+1$ hence $a^2+1\equiv 10 \pmod {72}$ and $a^2+3\equiv 12 \pmod {72}$. Hence, there exists a positive integer $c$ such that $a^2+1=2(36c+5)$ and $a^2+3=12(6c+1)$. Let $p_3$ be a prime factor of $36c+5$ and let $p_4$ be a prime factor of $6c+1$. It follows immediately that $p_1=3, p_2=2, p_3$, and $p_4$ are all distinct.
\end{proof}

We present a couple of particular instances covered by the above lemma.

Suppose first that $a=8$ and $b=1$. Then $u_2=a=2^3$ while $u_4=2^4\cdot 3\cdot 11$. Thus, in this case one can take $p_1=2, p_2=3, p_3=11$.

Suppose next that $a=-49$ and $b=1$. Then $u_2=a=-7^2$ while $u_4=-7^2\cdot 3^3\cdot 89$. It follows that we can choose $p_1=7, p_2=3, p_3=89$.

Finally, let us assume that $a=9$ and $b=1$. Then $u_2=a=3^2$ while $a^2+1=2\cdot 41$, $a^2+3=2^2\cdot 3\cdot 7$. Hence, in this case we have $p_1=3, p_2=2, p_3=41, p_4=7$.

The next lemma uses the concept of covering system introduced by Erd\H{o}s in \cite{erdos}.
\begin{df} A collection of residue classes $r_i \pmod{m_i}$, $0\le r_i<m_i$, where $1\le i\le t$ is said to be a \emph{covering system} if every integer belongs to at least one residue class $r_i\pmod{m_i}$.
\end{df}
In the proof of the theorem when $|b|=1$, except for finitely many values of $a$ we will use one of the following two covering systems
\begin{align*}
&0 \pmod{2}, \quad 1 \pmod{6},\quad 3\pmod{6},\quad 5\pmod{6}\\
&0 \pmod{2}, \quad 1\pmod{4}, \quad 3\pmod{4}.
\end{align*}
\begin{lemma}\label{last}
Let $a, b$ be two integers such that $|a|\ge 2$ and $|b|=1$. Let $(u_n)_{n\ge 0}$ be the Lucas sequence defined in \eqref{lucassequence} and  suppose that there exists a finite collection of triples $(p_i, m_i, r_i)$, $1\le i\le t$ with the following properties
\begin{itemize}
\item[$(i)$]{All primes $p_i$ are distinct.}
\item[$(ii)$]{The residue classes $r_i \pmod{m_i}$ form a covering system.}
\item[$(iii)$]{$p_i|u_{m_i}$ for all $1\le i\le t$.}
\end{itemize}
Then there exist two relatively prime positive integers $x_0$ and $x_1$ such that each term of the sequence $(x_n)_{n\ge 0}$ defined in \eqref{recurrence} is composite.
\end{lemma}
\begin{proof}
Let $P=p_1p_2\ldots p_t$. By the Chinese remainder Theorem, there exist $y, z \in \{0, 1, \ldots P-1\}$ satisfying
\begin{align}\label{congruences}
&y\equiv u_{m_i-r_i} \pmod{p_i},\notag\\
&z\equiv u_{m_i-r_i+1} \pmod{p_i}
\end{align}
\noindent for $i=1, 2,\ldots ,t $. Note that there is no prime which divides $y, z$, and $P$ simultaneously. Indeed, if such a prime $p_j$ were to exist then there will be two consecutive terms $u_n$ and $u_{n+1}$ both divisible by $p_j$. Since $u_{n+1}=au_n+bu_{n-1}$ and $|b|=1$ then $p_j| u_{n-1}$. Using induction it follows that $p_j|u_1$ which is impossible since $u_1=1$.

Let $x_0\equiv y\pmod{P}$ and $x_1\equiv z \pmod{P}$.

Then we have $x_0\equiv u_{m_i-r_i}\pmod{p_i}$ and $x_1\equiv u_{m_i-r_i+1}\pmod{p_i}$ for all $i=1,2,\ldots, t$.
By induction on $n$ we obtain that $x_{n+1}\equiv u_{m_i-r_i+n}\pmod{p_i}$ for every $n\ge 0$ and every $1\le i\le t$.

Since $\{r_i\pmod{m_i}\}_{i=1}^t$ is a covering system, each nonnegative integer $n$ belongs to one of these residue classes, say $n=r_i+km_i$ for some $k\ge 0$ and some $i\in \{1, 2,\ldots, t\}$. This implies that
\begin{equation}
x_{n+1}\equiv u_{m_i-r_i+n}\pmod{p_i}\equiv u_{m_i(k+1)}\pmod{p_i}\equiv 0\pmod{p_i},
\end{equation}
since $p_i|u_{m_i}$ and $u_{m_i}|u_{m_i(k+1)}$.
Hence, every term of the sequence $(x_n)_n$ is divisible by some prime $p_i$. It remains to choose $x_0$ and $x_1$ relatively prime positive integers $x_0\equiv y\pmod{P}$ and $x_1\equiv z\pmod{P}$ such that $|x_n|\ge P$ for every $n\in \mathbb{N}$.

In order to achieve this take $x_0=y+P$ and $x_1=z+kP$ where $k\ge 2$ and $\gcd(x_0, x_1)=1$. Using Lemma \ref{lemma3} with $n_1=y+P$, $n_2=z$ and $n_3=P$ shows that such a choice is always possible. Recall that we proved earlier that $\gcd(y, z, P)=1$.
Such a choice implies that $0<P\le x_0<x_1$ and since $|a|>|b|=1$ Lemma \ref{lemma2} implies that $(|x_n|)_n$ is a strictly increasing sequence. It follows that $|x_n|\ge P$ for all $n\ge 0$ and therefore each such $x_n$ is composite.
\end{proof}

We can now prove the theorem if $|b|=1$.

Suppose first that $b=-1$ and $|a|=p_1^s\ge 4$. Then by Lemma \ref{bminus1} there are four distinct primes $p_1, p_2, p_3, p_4$ dividing
$u_2, u_6, u_6, u_6$, respectively. The theorem follows after using lemma \ref{last} for the triples $(p_1, 2, 0), (p_2, 6, 1), (p_3, 6, 3), p_4, 6, 5)$.

As a numerical illustration suppose that $a=-9, b=-1$. Then, as described in the paragraph following the proof of Lemma \ref{bminus1} we have $p_1=3, p_2=2, p_3=5$ and $p_4=13$.

The system \eqref{congruences} becomes
\begin{align*}
&y\equiv u_2\pmod{3}\quad\quad\quad z\equiv u_3\pmod{3}\\
&y\equiv u_5\pmod{2}\quad\quad\quad z\equiv u_6\pmod{2}\\
&y\equiv u_3\pmod{5}\quad\quad\quad z\equiv u_4\pmod{5}\\
&y\equiv u_1\pmod{13}\quad\quad\,\,\, z\equiv u_2\pmod{13}
\end{align*}
and its solution is $P=p_1p_2p_3p_4=390, y=105, z=134$. Since $y<z$ and $\gcd(y,z)=1$ one can safely take $x_0=y=105$ and $x_1=z=134$. Then, the sequence $(|x_n|)_{n\ge 0}$ is strictly increasing and for every $n\ge 0$ we have that $x_{2n}\equiv 0 \pmod{3}, x_{6n+1}\equiv 0 \pmod{2}, x_{6n+3}\equiv 0\pmod{5}$ and $x_{6n+5}\equiv 0 \pmod{13}$. It follows that all terms of the sequence are composite.

Next suppose that $b=1$ and $|a|=p_1^s\ge 6$ for some prime $p_1$. If $p_1\neq 3$ then by Lemma \ref{bplus1} there exist three distinct primes $p_1, p_2, p_3$ dividing
$u_2, u_4, u_4$, respectively.

The theorem follows after using Lemma \ref{last} for the triples $(p_1, 2, 0), (p_2, 4, 1), (p_3, 4, 3)$.
As in the case $b=-1$ we present the details in a couple of particular cases.

Suppose first that $a=8, b=1$. Then, as described in the paragraph following the proof of Lemma \ref{bplus1} we have $p_1=2, p_2=3, p_3=11$.
The system \eqref{congruences} becomes
\begin{align*}
&y\equiv u_0\pmod{2}\quad\quad\quad z\equiv u_1\pmod{2}\\
&y\equiv u_3\pmod{3}\quad\quad\quad z\equiv u_4\pmod{3}\\
&y\equiv u_1\pmod{11}\quad\quad\,\,\, z\equiv u_2\pmod{11}
\end{align*}
and its solution is $P=p_1p_2p_3=66, y=56, z=63$. Note that in this case $\gcd(y,z)=7>1$.  Still, one can safely take $x_0=y=56$ and $x_1=z+P=129$ and now $\gcd(x_0,x_1)=1$ as desired. Then, since $0<x_0<x_1$ the sequence $(|x_n|)_{n\ge 0}$ is strictly increasing and for every $n\ge 0$ we have that $x_{2n}\equiv 0 \pmod{2}, x_{4n+1}\equiv 0 \pmod{3}$, and $x_{4n+3}\equiv 0 \pmod{11}$. It follows that all terms of the sequence are composite.

Next assume that $a=-49, b=1$. Then, $p_1=7, p_2=3, p_3=89$ and the system \eqref{congruences} becomes
\begin{align*}
&y\equiv u_0\pmod{7}\quad\quad\quad z\equiv u_1\pmod{7}\\
&y\equiv u_3\pmod{3}\quad\quad\quad z\equiv u_4\pmod{3}\\
&y\equiv u_1\pmod{89}\quad\quad\,\,\, z\equiv u_2\pmod{89}.
\end{align*}
The solution is $P=p_1p_2p_3=1869, y=980, z=1464$. Notice that in this case $\gcd(y,z)=4>1$.  Still, one can safely take $x_0=y=980$ and $x_1=z+P=3333$ and now $\gcd(x_0,x_1)=1$ as desired. Moreover, since $0<x_0<x_1$ the sequence $(|x_n|)_{n\ge 0}$ is strictly increasing and for every $n\ge 0$ we have that $x_{2n}\equiv 0 \pmod{7}, x_{4n+1}\equiv 0 \pmod{3}$, and $x_{4n+3}\equiv 0 \pmod{89}$. It follows that all terms of the sequence are composite.

For the case $b=1$ and $|a|=3^s$ with $s\ge 2$ we use the second part of Lemma \ref{bplus1} to conclude that there are four distinct primes $p_1, p_2, p_3, p_4$ dividing
$u_2, u_6, u_6, u_6$, respectively. The theorem follows after using lemma \ref{last} for the triples $(p_1, 2, 0), (p_2, 6, 1), (p_3, 6, 3), (p_4, 6, 5)$.

We show the full details if $a=9, b=1$. Then, as mentioned in the paragraph following the proof of Lemma \ref{bplus1} we have $p_1=3, p_2=2, p_3=41, p_4=7$.
The system \eqref{congruences} becomes
\begin{align*}
&y\equiv u_2\pmod{3}\quad\quad\quad z\equiv u_3\pmod{3}\\
&y\equiv u_5\pmod{2}\quad\quad\quad z\equiv u_6\pmod{2}\\
&y\equiv u_3\pmod{41}\quad\quad\,\,\, z\equiv u_4\pmod{41}\\
&y\equiv u_1\pmod{7}\quad\quad\quad z\equiv u_2\pmod{7}.
\end{align*}
Solving we obtain $P=p_1p_2p_3p_4=1722, y=1107, z=1444$. In this case we can simply choose $x_0=y=1107$ and $x_1=z=1144$.

Then,  $\gcd(x_0,x_1)=1$ and since $0<x_0<x_1$ the sequence $(|x_n|)_{n\ge 0}$ is strictly increasing. Moreover, for every $n\ge 0$ we have that $x_{2n}\equiv 0 \pmod{3}, x_{6n+1}\equiv 0 \pmod{2}, x_{6n+3}\equiv 0 \pmod{41}$, and $x_{6n+5}\equiv 0 \pmod{7}$. It follows that all terms of the sequence are composite as desired.

\bigskip
\bigskip

At this point we have proved the main theorem when $|b|=1$ for all but finitely values of $a$.
We still have to study what happens when $b=-1$ and $|a|\le 3$ as well as the cases when $b=1$ and $|a|\le 5$.
Recall that the cases $a=0$ and $(a,b)=(\pm 2, -1)$ were already handled in section 2.

For most of these cases we will still use Lemma \ref{last}; the only difference is that the set of triples $\{p_i,m_i,r_i\}_{i=1}^{i=t}$ is occasionally going to be slightly more numerous.

We summarize our findings in the table below. We invite the reader to verify that the collections $\{p_i,m_i,r_i\}_{i=1}^{i=t}$ do indeed satisfy the three conditions in Lemma \ref{last}. Note that in each case we have $0<x_0<x_1$ and since $|a|>|b|=1$ Lemma \ref{lemma2} implies that $(|x_n|)_{n\ge 0}$ is strictly increasing.
\begin{table}[htbp]
\centering
\begin{tabular}{c|c|c|c|c}
 \hline
 $a$ & $b$ & $\{(p_i,m_i,r_i)\}$ & $x_0$ & $x_1$\\ [0.5ex]
 \hline
 $\phantom{-}5$ & $\phantom{-}1$ & $(5, 2, 0), (2, 6, 1), (7, 6, 3), (13,6,5)$ & $495$ & $1136$ \\
 $-5$ & $\phantom{-}1$ & $(5, 2, 0), (2, 6, 1), (7, 6, 3), (13,6,5)$ & $495$ & $866$ \\
 $\phantom{-}4$ & $\phantom{-}1$ & $(2,2,0), (3,4,1), (7,8,3), (23,8,7)$ & $116$ & $165$ \\
 $-4$ & $\phantom{-}1$ & $(2,2,0), (3,4,1), (7,8,3), (23,8,7)$ & $116$ & $801$ \\
 $\phantom{-}3$ & $\phantom{-}1$ &$(3,2,0), (11,4,1), (7,8,3), (17,8,7)$  & $1803$ & $3454$ \\
 $-3$ & $\phantom{-}1$ &$(3,2,0), (11,4,1), (7,8,3), (17,8,7)$  & $1803$ & $3091$ \\
 $\phantom{-}2$ & $\phantom{-}1$ & $(2,2,0), (5,3,0), (3,4,1), (7,6,5), (11,12,7)$ & $260$ & $807$ \\
 $-2$ & $\phantom{-}1$ & $(2,2,0), (5,3,0), (3,4,1), (7,6,5), (11,12,7)$ & $260$ & $1503$ \\
 \hline
 $\phantom{-}3$ & $-1$ & $(3,2,0),(2,3,0),(7,4,3),(47,8,5),(23,12,5),(1103, 24, 1)$ & $7373556$ & $2006357$ \\
 $-3$ & $-1$ & $(3,2,0),(2,3,0),(7,4,3),(47,8,5),(23,12,5), (1103, 24,1)$ & $7373556$ & $14686445$ \\[1ex]
 \hline
\end{tabular}
\medskip
\caption{Covering triples for the cases $b=1, a=\pm 2. \pm 3, \pm 4, \pm 5$ and $b=-1, a=\pm 3$.}
\label{table1}
\end{table}

It remains to see what happens when $|a|=|b|=1$ as in these cases Lemma \ref{last} does not apply.

If $a=-1, b=-1$ then it can be easily verified that the sequence given by the recurrence $x_{n+1}=-x_n-x_{n-1}$ has period 3. Hence, if one chooses $x_0=8$ and $x_1=27$ then $x_2=-35$ and due to the periodic behavior all terms of the sequence are composite.

Similarly, if $a=1, b=-1$ then the sequence given by the recurrence $x_{n+1}=x_n-x_{n-1}$  has period 6. Again, if one chooses $x_0=8$ and $x_1=35$ then the first few terms of the sequence are: $8, 35, 27, -8, -35, -37, 8, 35, 27, \ldots$, that is, $x_n$ is always composite.

If $a=b=1$ then Vsemirnov's pair $v_0=106276436867$, $v_1=35256392432$ shows that all the numbers
\begin{equation}\label{v}
v_n=v_{n-1}+v_{n-2}=v_1F_n+v_0F_{n-1}.
\end{equation}
are composite. Here , $F_n$ is the $n$th Fibonacci number where we have $F_{-1}=1, F_0=0, F_1=1$.

For the case $a=-1, b=1$, we follow the solution in \cite{dubickas}.

It can be easily checked that the general term of the sequence $x_{n+1}=-x_n+x_{n-1}$ can be written as
\begin{equation}\label{x}
x_n=(-1)^{n+1}x_1F_n+(-1)^nx_0F_{n-1}, \quad n\ge 0.
\end{equation}
We choose $x_0=v_0-v_1=71020044435$ and $x_1=v_0=106276436867$. It is immediate that $x_0$ and $x_1$ are relatively prime composite integers.
Moreover, from \eqref{v} and \eqref{x} we obtain that
\begin{align*}
x_n&=(-1)^{n+1}v_0F_n+(-1)^n(v_0-v_1)F_{n-1}+= (-1)^{n+1}v_1F_{n-1}+(-1)^{n+1}v_0(F_n-F_{n-1})=\\
&=(-1)^{n+1}v_1F_{n-1}+(-1)^{n+1}v_0F_{n-2}= (-1)^{n+1}\left(v_1F_{n-1}+v_0F_{n-2}\right)=(-1)^{n-1}v_{n-1}.
\end{align*}
Hence, $|x_n|=v_{n-1}$ is composite for all $n\ge 0$.
The proof of the theorem is now complete.

\section{\bf A surprising result}

In this section we prove the following
\begin{thm}
Consider the integers $a, b$ such that $|a|\ge 3$ and $b=-1$. Let $u_0=0$, $u_1=1$ and $u_{n+1}=au_n+bu_{n-1}$ be the Lucas sequence of the first kind
associated to $a$ and $b$. Then, $|u_n|$ is composite for all $n\ge 3$.
\end{thm}
\begin{proof}
One has $u_3=a^2-1$, and $u_4=a^3-2a$ which are obviously composite. Since $|a|>|b|=1$, Lemma \ref{lemma2} implies that the sequence $(|u_n|)_{n\ge 0}$ is strictly increasing.
Suppose for the sake of contradiction that there exists an $n\ge 2$ such that $|u_{n+1}|=p$ where $p$ is some prime number. Since $|u_{n+1}|\ge |u_3|=a^2-1$ it follows that necessarily $p>|a|$.

Using now Lemma \ref{lemma1} for the sequence $(u_n)_{n\ge 0}$, equality \eqref{eq1} becomes
\begin{equation*}
u^2_{n+1}-au_nu_{n+1}-bu_n^2=(-b)^n(u_1^2-au_1u_0-bu_0^2),
\end{equation*}
and since $u_0=0$, $u_1=1$, $b=-1$, and $|u_{n+1}|=p$ we obtain that
\begin{equation}\label{quadratic}
u_n^2\pm apu_n+p^2-1=0.
\end{equation}
Regard the above equation as a quadratic in $u_n$. Since $u_n\in \mathbb{Z}$ it is necessary that the discriminant is a perfect square, that is, there exist a nonnegative integer $c$ such that
\begin{equation}\label{eq1000}
a^2p^2-4(p^2-1)=c^2 \quad \text{from which}\quad (a^2-4)p^2=c^2-4=(c-2)(c+2).
\end{equation}
Since $|a|\ge 3$ one can assume that $c\ge 3$. Since $p$ is a prime and $p^2$ divides $(c-2)(c+2)$ we have two possibilities. If $p$ divides both $c-2$ and $c+2$ then $p$ divides $4$ which means that $p=2$. However, this is impossible since $p>|a|\ge 3$.
Otherwise, $p^2$ divides either $c-2$ or $c+2$. In either case we obtain that $c+2\ge p^2$.
Using this inequality in \eqref{eq1000} it follows that
\begin{equation*}
(a^2-4)p^2=(c-2)(c+2)\ge p^2(p^2-4)\longrightarrow |a|\ge p,\quad \text{contradiction}.
\end{equation*}
This completes the proof.
\end{proof}

What is remarkable about this situation is that while $u_n$ is composite for every $n\ge 3$ it appears that there is no finite set of primes $p_1, p_2,\ldots p_t$ such that every
$u_n$ is divisible by some $p_i$, $1\le i\le t$. In fact, we suspect the following is true.
\begin{conj}
Let $(u_n)_{n\ge 0}$ be the Lucas sequence associated to some $|a|\ge 3$ and $b=-1$. Then for any two different primes $p$ and $q$ we have that $u_p$ and $u_q$ are relatively prime.
\end{conj}

If true, this conjecture would immediately imply that the set of prime factors of $\{u_n \, | n\ge 0\}$ is infinite.

\end{document}